\documentclass[11pt,dvips,twoside]{article}

\usepackage{pslatex}
\usepackage{fancyhdr}
\usepackage{graphicx}
\usepackage{geometry}

\RequirePackage{amsfonts,amssymb,amsmath,amscd,amsthm}
\RequirePackage{txfonts}
\RequirePackage{graphicx}
\RequirePackage{xcolor}
\RequirePackage{geometry}
\RequirePackage{enumerate}

\def\figurename{Figure} 
\makeatletter
\renewcommand{\fnum@figure}[1]{\figurename~\thefigure.}
\makeatother

\def\tablename{Table} 
\makeatletter
\renewcommand{\fnum@table}[1]{\tablename~\thetable.}
\makeatother

\usepackage{amsmath}
\usepackage{amssymb}
\usepackage{amsfonts}
\usepackage{amsthm,amscd}

\newtheorem{theorem}{Theorem}[section]
\newtheorem{lemma}[theorem]{Lemma}
\newtheorem{corollary}[theorem]{Corollary}
\newtheorem{proposition}[theorem]{Proposition}

\theoremstyle{example}

\theoremstyle{definition}
\newtheorem{definition}[theorem]{Definition}

\theoremstyle{remark}
\newtheorem{remark}[theorem]{Remark}

\numberwithin{equation}{section}

\begin{document}
\title{\bfseries\scshape{Classification, Derivations and Centroids of Low-Dimensional Real Trialgebras}}

\author{\bfseries\scshape Ahmed Zahari\thanks{e-mail address: zaharymaths@gmail.com}\\
Universit\'{e} de Haute Alsace,\\
 IRIMAS-D\'{e}partement de Math\'{e}matiques,\\
 18, rue des Fr\`eres Lumi\`ere F-68093 Mulhouse, France.\\
\bfseries\scshape Bouzid Mosbahi\thanks{e-mail address: mosbahibouzid@yahoo.fr}\\
 University of Sfax, Faculty of Sciences of Sfax,  BP 1171, 3000 Sfax, Tunisia.\\
 \bfseries\scshape Imed Basdouri\thanks{e-mail address: basdourimed@yahoo.fr}\\
 University of Gafsa, Faculty of Sciences of Gafsa, 2112 Gafsa, Tunisia.
}

\date{}
\maketitle


\noindent\hrulefill

\noindent {\bf Abstract.}
In this paper we study the structure and the algebraic varieties of associative trialgebras. We provide a classification of $n$-dimensional
associative trialgebras for $n\leq4$. Using the classification result of associative trialgebras, we describe the derivations and centroids
of low-dimensional associative trialgebras. We review some proprieties of the centroids in light of associative trialgebras and and we calculate the centroids
of low-dimensional associative trialgebras.

\noindent \hrulefill

\vspace{.3in}


\vspace{.08in} \noindent \textbf{Keywords}:
 Associative trialgebra,  Classification, Derivation, Central Derivation, Centroid, Rota-Baxter trialgebra.
\vspace{.3in}
\vspace{.2in}

\pagestyle{fancy} \fancyhead{} \fancyhead[EC]{ }
\fancyhead[EL,OR]{\thepage} \fancyhead[OC]{Ahmed Zahari,  Bouzid Mosbahi, Imed Basdouri} \fancyfoot{}
\renewcommand\headrulewidth{0.5pt}

\section{Introduction}

An associative trialgebra (or triassociative algebra) $(\mathcal{T}, \dashv, \vdash,\bot)$ is consisted of a vector space, three multiplications. It may be viewed as a deformation of an associative algebras. The associative trialgebras degenerate to exactly an associative trialgebras. In this paper, we aim  to study the structure of associative trialgebras.
Let $\mathcal{T}$ be a $n$-dimensional $\mathbb{K}$-linear space and $\left\{e_1, e_2, \cdots, e_n\right\}$ be a basis of $\mathcal{T}$.
A triassociative structure on $\mathcal{T}$ with product $\gamma, \delta$ and $\phi$ are determined by $3n^3$ structure constants $\mathcal{\gamma}_{ij}^k, \mathcal{\delta}_{ij}^k$ and $\mathcal{\phi}_{ij}^k$,  were $e_i\dashv e_j=\sum_{k=1}^n\gamma_{ij}^ke_k,\quad e_i\vdash e_j=\sum_{k=1}^n\delta_{ij}^ke_k$
and  $e_i\bot e_j=\sum_{k=1}^n\phi_{ij}^ke_k$. Requiring the algebra structure to be triassociative and unital gives rise to sub-variety $\mathcal{T}_t$ of $\mathbb{K}^{3n^3}$.
Basic changes in $\mathcal{T}$ result in the natural transport of structure action of $GL_n(k)$ on $\mathcal{T}_t$. Thus, isomorphism classes of $n$-dimensional algebras are one-to-one correspondence with the orbits of the action of $GL_n(k)$ on $\mathcal{T}_t$.

Classification problems of the associative trialgebras using the algebraic  and geometric technique prompted interest in the derivations and centroids of associative trialgebras. The associative trialgebras introduced by Loday \cite{Ld} with a motivation to provide dual of dialgebras, and have been further studied with  connections to several areas in mathematics
and physics.

In this paper our interest  is to study the derivations and centroids of finite dimensional associative trialgebras. The algebra of derivations and centroids are very useful in
algebraic and geometric classification problems of algebras.

The content of the present paper section-wise can be described as follows. In the first section, we introduce the subject alongside with some previously obtained results.
The goal of this paper is to introduce and classify derivations  and centroids of associative trialgebras. The paper is organized as follows.
In section 2, we provide some basic concepts needed for this study.
Section 3 is about the algebraic varieties of associative trialgebras, and we provide classifications, up to isomorphism, of two-dimensional, three-dimensional
and four-associative trialgebras.
In Section 4, we  give the classification of the derivations. Finally, in Section 5, we  give the classification of the centroids. The concept of derivations and  centroids in this case is easily
militated from that of finite-dimensional algebras. The algebra of centroids play important role in the classification problems and in different applications
of algebras.  In the study we make use of the classification results of two, three and four-dimensional associative trialgebras. All algebras and vectors spaces considered are supposed to
over a field $\mathbb{K}$ of characteristic zero.

\section{Preliminaries}
\begin{definition}\label{dia}
An associative dialgebra is a vector space $\mathcal{D}$ equipped with two binary operations : $\dashv$ called left and $\vdash$ called right,\\
(left)$\quad\dashv  : \, \mathcal{D}\times \mathcal{D}\rightarrow \mathcal{D}\quad$ and $\quad\quad$ (right)$\quad\vdash :\, \mathcal{D}\times \mathcal{D}\rightarrow \mathcal{D}$
satisfying the relations

\begin{eqnarray}
(x\dashv y)\dashv z&=&x\dashv(y\dashv z),\label{eq1}\\
(x\dashv y)\dashv z&=&x\dashv(y\vdash z),\label{eq2}\\
(x\vdash y)\dashv z&=&x\vdash(y\dashv z),\label{eq3}\\
(x\dashv y)\vdash z&=&x\vdash(y\vdash z),\label{eq4}\\
(x\vdash y)\vdash z&=&x\vdash(y\vdash z).\label{eq5}
\end{eqnarray}
\end{definition}

\begin{definition}
An associative trialgebra is a $k$-vector space $(\mathcal{T}, \bot, \dashv, \vdash)$ such that $(\mathcal{T},  \dashv, \vdash)$ is a associative dialgebra, $(\mathcal{T}, \bot)$
an associative algebra and,

\begin{eqnarray}
(x\dashv y)\dashv z&=&x\dashv(y\dashv z),\label{eq6}\\
(x\dashv y)\dashv z&=&x\dashv(y\vdash z),\label{eq7}\\
(x\vdash y)\dashv z&=&x\vdash(y\dashv z),\label{eq8}\\
(x\dashv y)\vdash z&=&x\vdash(y\vdash z),\label{eq9}\\
(x\vdash y)\vdash z&=&x\vdash(y\vdash z),\label{eq10}\\
(x\dashv y)\dashv z&=&x\dashv(y\bot z),\label{eq11}\\
(x\bot y)\dashv z&=&x\bot(y\dashv z),\label{eq12}\\
(x\dashv y)\bot z&=&x\bot(y\vdash z),\label{eq13}\\
(x\vdash y)\bot z&=&x\vdash(y\bot z),\label{eq14}\\
(x\bot y)\vdash z&=&x\vdash(y\vdash z),\label{eq15}\\
(x\bot y)\bot z&=&x\bot(y\bot z).\label{eq16}
\end{eqnarray}
\end{definition}

\begin{definition}
Let $(\mathcal{T}_1, \bot_1, \dashv_1, \vdash_1)$ , $(\mathcal{T}_2, \bot_2, \dashv_2, \vdash_2)$ be  associative trialgebras over a fied $\mathbb{K}$. Then a homomorphism
from $\mathcal{T}_1$ to $\mathcal{T}_2$ is a $\mathbb{K}$-linear mapping $\psi : \mathcal{T}_1\longrightarrow \mathcal{T}_2$ such that
\begin{eqnarray}
\psi(\dashv_1y)&=&\psi(x)\dashv_2\psi(y)\\
\psi(\vdash_1y)&=&\psi(x)\vdash_2\psi(y)\\
\psi(\bot_1y)&=&\psi(x)\bot_2\psi(y)
\end{eqnarray}
for all $x, y\in \mathcal{T}_1.$
\end{definition}
\begin{remark}
A bijective homomorphism is an isomorphism of $\mathcal{T}_1$ and $\mathcal{T}_2$.
\end{remark}

\begin{proposition}
Let $(\mathcal{T},\dashv,\bot,\vdash)$ be an associative trialgebras. Then $\mathcal{T}$ is an associative algebra with
respect to the multiplication $\ast : \mathcal{T} \otimes \mathcal{T} \longrightarrow \mathcal{T} :$
\begin{center}
$x \ast y = x \dashv y + x \vdash y - x \bot y$
\end{center}
for any $x, y \in \mathcal{T}.$
\end{proposition}

\begin{proof}
Using the axioms of associative trialgebra we have for $x,y \in \mathcal{T}$\\
$(x\ast y) \ast z = (x\dashv y+x\vdash y-x\bot y)\ast z$\\
$=(x\dashv y+ x\vdash y-x\bot y)\dashv z+(x\dashv y +x\vdash y-x\bot y)\vdash z-(x\dashv y +x\vdash y-x\bot y)\bot z$\\
$=(x\dashv y)\dashv z+(x\vdash y)\dashv z-(x\bot y)\dashv z+(x\dashv y)\vdash z+(x\vdash y)\vdash z-(x\bot y)\vdash z- (x\dashv y)\bot z-(x\vdash y)\bot z+(x\bot y)\bot z$\\
$=x\dashv (y\vdash z)+x\vdash (y\dashv z)-x\bot (y\dashv z)+x\vdash (y\vdash z)+x\vdash (y\vdash z)-x\vdash (y\vdash z)-x\bot (y\vdash z)-x\vdash (y\bot z)+x\bot (y\bot z)$\\
$=x\dashv (y\ast z-y\dashv z+y\bot z)+x\vdash (y\dashv z)-x\bot (y\dashv z)+x\vdash (y\vdash z)-x\bot (y\vdash z)-x\vdash (y\bot z)+x\bot (y\bot z)$\\
$=x\dashv (y\ast z)-x\dashv(y\dashv z)+x\dashv(y\bot z)+x\vdash (y\dashv z)-x\bot (y\dashv z)+x\vdash (y\vdash z)-x\bot (y\vdash z)-x\vdash (y\bot z)+x\bot (y\bot z)$\\
$=x\vdash(y\dashv z+y\vdash z-y\bot z)-x\bot (y\dashv z+y\vdash z-y\bot z)+x\dashv(y\ast z)$\\
$=x\dashv(y\ast z)+x\vdash(y\ast z)-x\bot(y\ast z)$\\
$=x\ast(y\ast z).$
\end{proof}

\begin{definition}
Let $\mathcal{A}$ be a $\mathbb{K}$-algebra and let $\lambda \in \mathbb{K}$. If a $\mathbb{K}$-linear map $R : \mathcal{A} \longrightarrow \mathcal{A}$ satisfies
the Rota-Baxter relation:
\begin{center}
 $R(x)R(y) = R(R(x)y + xR(y) + \lambda xy)$
 \end{center}
 $\forall x, y \in \mathcal{A}$, then $R$ is called a Rota-Baxter operator of weight $\lambda$ and $(\mathcal{A}, R)$ is called a Rota-Baxter algebra
of weight $\lambda$.
\end{definition}


\begin{remark}
If $R$ is a Rota-Baxter operator of weight $\lambda \in \mathbb{K}$ on trialgebras $(\mathcal{T},\dashv,\bot,\vdash)$. It is also
a Rota-Baxter operator of weight $\lambda \in \mathbb{K}$ on the associative algebra $(\mathcal{T},\ast)$.
\end{remark}

\begin{proposition}
Let $(\mathcal{T},\dashv,\bot,\vdash)$ be a Rota-Baxter trialgebras of weight $0$. Then $(\mathcal{T},\star)$ is a left-symmetric algebra with
\begin{center}
$x \star y = R(x)  \ast y - y  \ast R(x) \qquad and \qquad x  \ast y = x \dashv y + x \vdash y - x \bot y$
\end{center}
for all $x, y \in \mathcal{T}.$
\end{proposition}

\begin{proof}
For $x,y \in \mathcal{T}$ we have
$$\begin{array}{ll}
(x \star y)\star z
&=(R(x) \ast y - y \ast R(x)) \star z\\
&=R(R(x) \ast y - y \ast R(x)) \ast z - z \ast R(R(x) \ast y - y \ast R(x))\\
&=R(R(x) \ast y) \ast z - R(y \ast R(x)) \ast z - z \ast R(R(x) \ast y) + z \ast R(y \ast R(x))
\end{array}$$
and
$$\begin{array}{ll}
x \star (y \star z)
&= x \star (R(y) \ast z - z \ast R(y))\\
&= R(x) \ast (R(y) \ast z - z \ast R(y)) - (R(y) \ast z - z \ast R(y)) \ast R(x)\\
&=R(x) \ast (R(y) \ast z) - R(x) \ast (z \ast R(y)) - (R(y) \ast z) \ast R(x) + (z \ast R(y)) \ast R(x)
\end{array}.$$
Then
$$\begin{array}{ll}
(x\star y)\star z - x\star (y \star z)
&-(y\star x)\star z+y\star (x \star z)\\
&= R(R(x) \ast y) \ast z - R(y \ast R(x)) \ast z - z \ast R(R(x) \ast y) + z \ast R(y \ast R(x))\\
&-R(x) \ast (R(y) \ast z) + R(x) \ast (z \ast R(y)) + (R(y) \ast z) \ast R(x) - (z \ast R(y)) \ast R(x)\\
&-R(R(y) \ast x) \ast z + R(x \ast R(y)) \ast z + z \ast R(R(y) \ast x) - z \ast R(x \ast R(y))\\
&+R(y) \ast (R(x) \ast z) - R(y) \ast (z \ast R(x)) - (R(x) \ast z) \ast R(y) + (z \ast R(x)) \ast R(y).\\
\end{array}.$$

Using $x \ast y = x \dashv y + x \vdash y - x \bot y$ and the Rota-Baxter identities.Then associativity leads to\\
$(x\star y)\star z - x\star (y \star z)-(y\star x)\star z+y\star (x \star z)=0.$  Therefore we obtain $(x,y,z)=(y,x,z).$
\end{proof}

\begin{proposition}
Let $(A,\dashv,\bot,\vdash,R)$ be a Rota-Baxter trialgebras of weight $-1$. Then $(A,\star)$ is an associative algebra with
\begin{center}
$x \star y = R(x) \ast y - y \ast R(x) - x \ast y \qquad and \qquad x \ast y = x \dashv y + x \vdash y - x \bot y.$
\end{center}
\end{proposition}

\begin{proof}
For $x,y \in \mathcal{T}$ we have

$$\begin{array}{ll}
x \star (y \star z)
&= R(x) \ast (R(y) \ast z - z \ast R(y) - y \ast z)\\
& - (R(y) \ast z-z\ast R(y)-y\ast z) \ast R(x)-x \ast(R(y) \ast z - z \ast R(y) - y \ast z)
\end{array}$$
and
$$\begin{array}{ll}
(x \star y) \star z
&=R(R(x) \ast y-y\ast R(x)-x \ast y)\ast z\\
&-z \ast R(R(x) \ast y-y \ast R(x)-x\ast y)-(R(x) \ast y-y \ast R(x)- x \ast y) \ast z
\end{array}$$
Then we obtain
$$\begin{array}{ll}
x \star (y \star z) - (x \star y) \star z
&=R(x) \ast (R(y) \ast z - z \ast R(y) - y \ast z) - (R(y) \ast z - z \ast R(y) - y \ast z) \ast R(x)\\
&- x \ast (R(y) \ast z - z \ast R(y) - y \ast z)-R(R(x) \ast y + y \ast R(x) + x \ast y) \ast z\\
&+ z \ast R(R(x) \ast y + y \ast R(x) + x \ast y)+ (R(x) \ast y + y \ast R(x) + x \ast y) \ast z
\end{array}$$
Then it vanishes using $x \ast y = x \dashv y + x \vdash y - x \bot y$ and the Rota-Baxter identities.
\end{proof}

\section{Classification of low-dimensional associative trialgebras}
The classification problem of algebra is one of the important problems of modern algebras.

In this section we recall some elementary facts on triassociative algebras that will be used later on.
Let V be an n-dimensional vector space and $\left\{e_1, e_2, \cdots, e_n\right\}$ be a basis of V. Then a triassociative structure on V can be defined as three bilinear mappings:
$$\lambda : V \times V \longrightarrow V$$
 representing the left product $\dashv$,
$$\mu : V \times V \longrightarrow V$$
representing the right product $\vdash$, and
$$\xi : V \times V \longrightarrow V$$
 representing the middle product $\perp$ , consented via triassociative algebra axioms.
Hence, an n-dimensional triassociative algebra $\mathcal{T}$ can be seen as a
triple $\mathcal{T} = (V,\lambda,\mu,\xi)$ where $\lambda , \mu$ and $\xi$ are associative laws on V. We will
denote by Trias the set of triassociative algebra laws on V.\\
Let us denote by $\mathcal{\gamma}_{ij}^k, \mathcal{\delta}_{rs}^t$ and $\mathcal{\phi}_{pl}^m$, where, $i, j,k,r,s,t,p,l,m$ the structure constants of a triassociative algebra with respect to the basis $\left\{e_1, e_2, \cdots, e_n\right\}$ of V, where
$$e_i\dashv e_j=\sum_{k=1}^n\gamma_{ij}^ke_k,\quad e_r\vdash e_s=\sum_{t=1}^n\delta_{rs}^te_t,\quad
\text{and}\quad e_p\bot e_l=\sum_{m=1}^n\phi_{pl}^me_m,\quad for\quad i,j, k, q, r, s,t, p, l\in \left\{1,n\right\}.$$
Then Trias can be considered as a closed subset of $3n^3$-dimensional
affine space specified by the following system of polynomial equations with
respect to the structure constants $\mathcal{\gamma}_{ij}^k, \mathcal{\delta}_{rs}^t$ and $\mathcal{\phi}_{pl}^m$:
$$\left\{\begin{array}{c}
\begin{array}{ll}
\sum_{p=1}^n(\gamma_{ij}^p\gamma_{pk}^q-\gamma_{jk}^p\gamma_{ip}^q)=0,\quad i,j,q\in \left\{1,n\right\}\\
\sum_{p=1}^n(\gamma_{ij}^p\gamma_{pk}^q-\delta_{jk}^p\gamma_{ip}^q)=0,\quad i,j,q\in \left\{1,n\right\}\\
\sum_{p=1}^n(\delta_{ij}^p\gamma_{pk}^q-\gamma_{jk}^p\delta_{ip}^q)=0,\quad i,j,q\in \left\{1,n\right\}\\
\sum_{p=1}^n(\gamma_{ij}^p\delta_{pk}^q-\delta_{jk}^p\delta_{ip}^q)=0,\quad i,j,q\in \left\{1,n\right\}\\
\sum_{p=1}^n(\delta_{ij}^p\delta_{pk}^q-\delta_{jk}^p\delta_{ip}^q)=0,\quad i,j,q\in \left\{1,n\right\}\\
\sum_{p=1}^n(\gamma_{ij}^p\gamma_{pk}^q-\phi_{jk}^p\gamma_{ip}^q)=0,\quad i,j,q\in \left\{1,n\right\}\\
\end{array}
\end{array}\right.
\left\{\begin{array}{c}
\begin{array}{ll}
\sum_{p=1}^n(\phi_{ij}^p\gamma_{pk}^q-\gamma_{jk}^p\phi_{ip}^q)=0,\quad i,j,q\in \left\{1,n\right\}\\
\sum_{p=1}^n(\gamma_{ij}^p\phi_{pk}^q-\delta_{jk}^p\phi_{ip}^q)=0,\quad i,j,q\in \left\{1,n\right\}\\
\sum_{p=1}^n(\delta_{ij}^p\phi_{pk}^q-\phi_{jk}^p\delta_{ip}^q)=0,\quad i,j,q\in \left\{1,n\right\}\\
\sum_{p=1}^n(\phi_{ij}^p\delta_{pk}^q-\delta_{jk}^p\delta_{ip}^q)=0,\quad i,j,q\in \left\{1,n\right\}\\
\sum_{p=1}^n(\phi_{ij}^p\phi_{pk}^q-\phi_{jk}^p\phi_{ip}^q)=0,\quad i,j,q\in \left\{1,n\right\}.
\end{array}
\end{array}\right.$$
Thus Trias can be considered as a subvariety of $3n^3$-dimensional affine space. On Trias the linear matrix group  $GL_n$
 acts by changing of basis.

\begin{lemma}
The axioms in Definition \ref{dia} are respectively equivalent to
$$\left\{\begin{array}{c}
\begin{array}{ll}
\gamma_{ij}^p\gamma_{pk}^q-\gamma_{jk}^p\gamma_{ip}^q=0,\quad i,j,q\in \left\{1,n\right\}\\
\gamma_{ij}^p\gamma_{pk}^q-\delta_{jk}^p\gamma_{ip}^q=0,\quad i,j,q\in \left\{1,n\right\}\\
\delta_{ij}^p\gamma_{pk}^q-\gamma_{jk}^p\delta_{ip}^q=0,\quad i,j,q\in \left\{1,n\right\}\\
\gamma_{ij}^p\delta_{pk}^q-\delta_{jk}^p\delta_{ip}^q=0,\quad i,j,q\in \left\{1,n\right\}\\
\delta_{ij}^p\delta_{pk}^q-\delta_{jk}^p\delta_{ip}^q=0,\quad i,j,q\in \left\{1,n\right\}\\
\gamma_{ij}^p\gamma_{pk}^q-\phi_{jk}^p\gamma_{ip}^q=0,\quad i,j,q\in \left\{1,n\right\}\\
\end{array}
\end{array}\right.
\left\{\begin{array}{c}
\begin{array}{ll}
\phi_{ij}^p\gamma_{pk}^q-\gamma_{jk}^p\phi_{ip}^q=0,\quad i,j,q\in \left\{1,n\right\}\\
\gamma_{ij}^p\phi_{pk}^q-\delta_{jk}^p\phi_{ip}^q=0,\quad i,j,q\in \left\{1,n\right\}\\
\delta_{ij}^p\phi_{pk}^q-\phi_{jk}^p\delta_{ip}^q=0,\quad i,j,q\in \left\{1,n\right\}\\
\phi_{ij}^p\delta_{pk}^q-\delta_{jk}^p\delta_{ip}^q=0,\quad i,j,q\in \left\{1,n\right\}\\
\phi_{ij}^p\phi_{pk}^q-\phi_{jk}^p\phi_{ip}^q=0,\quad i,j,q\in \left\{1,n\right\}.
\end{array}
\end{array}\right.$$
\end{lemma}
Note that ${Trias}_n^m$ denote $m^{th}$ isomorphism class of associative trialgebra in dimension $n.$
\begin{theorem}\label{the1}
 Any 2-dimensional  real associative trialgebra either is associative or isomorphic to one of the following pairwise non-isomorphic triassociative algebras:
\end{theorem}

${Trias}_2^1$ :	
$\begin{array}{ll}
e_1\dashv e_2=ae_1,\\
e_2\dashv e_2=ae_2,\\
\end{array}\quad$
$\begin{array}{ll}
e_2\vdash e_1=ae_1,\\
e_2\vdash e_2=ae_2,\\
\end{array}\quad$
$\begin{array}{ll}
e_1\bot e_1=be_1,\\
e_1\bot e_2=be_1+ae_2,
\end{array}$\\

${Trias}_2^2$ :	
$\begin{array}{ll}
e_1\dashv e_1=e_1,\\
e_2\dashv e_1=e_2,\\
\end{array}\quad$
$\begin{array}{ll}
e_1\vdash e_1=e_1,\\
e_1\vdash e_2=e_2,
\end{array}\quad\quad$
$\begin{array}{ll}
e_1\bot e_2=e_1+ae_2,\\
e_2\bot e_2=e_2.
\end{array}$\\

${Trias}_2^3$ :	
$\begin{array}{ll}
e_2\dashv e_2=e_2,\\
\end{array}\quad$
$\begin{array}{ll}
e_2\vdash e_1=e_1,\\
e_2\vdash e_2=e_2,\\
\end{array}\quad\quad$
$\begin{array}{ll}
e_2\bot e_1=e_1,\\
e_2\bot e_2=e_2.
\end{array}$\\

${Trias}_2^4$ :	
$\begin{array}{ll}
e_1\dashv e_2=e_1,\\
e_2\dashv e_2=e_2,\\
\end{array}\quad$
$\begin{array}{ll}
e_2\vdash e_2=e_1+e_2,
\end{array}\,\,$
$\begin{array}{ll}
e_2\bot e_2=e_1+e_2.
\end{array}$\\

${Trias}_2^5$ :	
$\begin{array}{ll}
e_1\dashv e_1=e_1,\\
\end{array}\quad$
$\begin{array}{ll}
e_2\vdash e_1=e_1,\\
e_2\vdash e_2=e_2,
\end{array}\quad\quad$
$\begin{array}{ll}
e_1\bot e_1=e_1,\\
e_1\bot e_2=e_2.
\end{array}$\\

${Trias}_2^6$ :	
$\begin{array}{ll}
e_1\dashv e_1=e_1,\\
e_2\dashv e_1=e_2,\\
\end{array}\quad$
$\begin{array}{ll}
e_1\vdash e_1=e_1,
\end{array}\quad\quad$
$\begin{array}{ll}
e_1\bot e_1=e_1.
\end{array}$\\

${Trias}_2^7$ :	
$\begin{array}{ll}
e_1\dashv e_1=e_1,\\
e_2\dashv e_1=e_2,\\
\end{array}\quad$
$\begin{array}{ll}
e_1\vdash e_1=e_1,\\
e_1\vdash e_2=e_2,
\end{array}\quad\quad$
$\begin{array}{ll}
e_1\bot e_1=e_1,\\
e_1\bot e_2=e_2.\\
\end{array}$\\

${Trias}_2^8$ :	
$\begin{array}{ll}
e_1\dashv e_1=ae_1,\\
e_2\dashv e_1=ae_2,\\
\end{array}\,\,$
$\begin{array}{ll}
e_1\vdash e_1=ae_1,\\
e_1\vdash e_2=ae_2,
\end{array}\quad$
$\begin{array}{ll}
e_1\bot e_1=ae_1+be_2.
\end{array}$

\begin{proof}
Let $\mathcal{T}$ be a two-dimensional vector space. To determine an associative trialgebras structure on $\mathcal{T}$ , we consider $\mathcal{T}$ with respect to one associative trialgebra
operation. Let $\mathcal{A}=(\mathcal{T}, \dashv)$ be the algebra
$$e_1 \dashv e_1 = e_1, \: e_2 \dashv e_1 = e_2$$
The multiplication operations $\vdash,\bot \:in\: \mathcal{T}$ , we define as follows:

$$\begin{array}{ll}
e_1\vdash e_1 = \alpha_1e_1+\alpha_2e_2,\\
e_1 \vdash e_2 =\alpha_3e_1+\alpha_4e_2,\\
e_2\vdash e_1 =\alpha_5e_1+\alpha_6e_2,\\
\end{array}
\begin{array}{ll}
e_2\vdash e_2 =\alpha_7e_1+\alpha_8e_2,\\
e_1\bot e_1 = \beta_1e_1+\beta_2e_2,\\
e_1\bot e_2 = \beta_3e_1+\beta_4e_2,\\
\end{array}
\begin{array}{ll}
e_2\bot e_1 = \beta_5e_1+\beta_6e_2,\\
e_2\bot e_2 = \beta_7e_1+\beta_8e_2.
\end{array}$$

Now verifying associative trialgebra axioms, we get several constraints for the coefficients $\alpha_i,\beta_i$ where $1\leq i\leq 8.$\\
Applying $(e_1 \vdash e_1)\dashv e_1 = e_1\vdash (e_1\dashv e_1)$, we get $(\alpha_1e_1+\alpha_2e_2)\dashv e_1 = e_1\vdash e_1$ and then $\alpha_1e_1 = e_1$. Therefore $\alpha_1 = 1$.
The verification, $(e_1 \vdash e_1) \dashv e_1 = e_1 \vdash (e_1 \vdash e_1)$ leads to $(e_1 + \alpha_2e_2) \vdash e_1 = e_1 \vdash e_1$. We have $e_1 +\alpha_2e_2 = e_1$.
Hence we have $\alpha_2 = 0$. Consider $(e_1 \bot e_1) \dashv e_1 = e_1 \bot (e_1 \dashv e_1)$. It implies that $(\beta_1e_1 + \beta_2e_2) \dashv e_1 = e_1 \bot e_1$,
therefore $\beta_1 = 1$ and $\beta_2 = 0$.
then $\mathcal{A}=(\mathcal{T}, \dashv)$ it is isomorphic to ${Trias}_2^6$.
The other associative trialgebras of the list of Theorem \ref{the1} can be obtained by minor modifications of the observation above.
\end{proof}

\begin{theorem}\label{the2}
 Any 3-dimensional real associative trialgebra either is associative or isomorphic to one of the following  pairwise non-isomorphic associative trialgebras :
\end{theorem}

${Trias}_{3}^1$ :	
$\begin{array}{ll}
e_1\dashv e_2=e_3,\\
e_2\dashv e_1=e_3,\\
\end{array}\,\,$
$\begin{array}{ll}
e_2\dashv e_3=e_3,\\
e_1\vdash e_2=e_3,\\
\end{array}\quad$
$\begin{array}{ll}
e_2\vdash e_2=e_3,
e_1\bot e_1=e_3,\\
\end{array}$
$\begin{array}{ll}
e_1\bot e_2=e_3,\\
e_2\bot e_2=e_3.
\end{array}$

${Trias}_{3}^2$ :	
$\begin{array}{ll}
e_1\dashv e_2=e_3,\\
e_2\dashv e_1=e_3,\\
e_2\dashv e_3=e_3,\\
\end{array}\,\,$
$\begin{array}{ll}
e_1\vdash e_2=e_3,\\
e_2\vdash e_1=e_3,\\
\end{array}\quad$
$\begin{array}{ll}
e_2\vdash e_2=e_3\\
e_1\bot e_1=e_3,\\
\end{array}$
$\begin{array}{ll}
e_1\bot e_2=e_3,\\
e_2\bot e_2=e_3.
\end{array}$

${Trias}_{3}^3$ :	
$\begin{array}{ll}
e_2\dashv e_2=e_1,\\
\end{array}\,\,$
$\begin{array}{ll}
e_2\vdash e_2=e_1,\\
\end{array}\quad$
$\begin{array}{ll}
e_2\bot e_2=e_3,\\
e_2\bot e_3=e_1+e_3.
\end{array}$

${Trias}_{3}^4$ :	
$\begin{array}{ll}
e_3\dashv e_3=e_1,\\
\end{array}\,\,$
$\begin{array}{ll}
e_3\vdash e_3=e_1,\\
\end{array}\quad$
$\begin{array}{ll}
e_3\bot e_2=e_1+e_2,\\
e_3\bot e_3=e_1+e_2.
\end{array}$

${Trias}_{3}^5$ :	
$\begin{array}{ll}
e_2\dashv e_2=e_1+e_3,
\end{array}\,\,$
$\begin{array}{ll}
e_2\vdash e_2=e_1+e_3,
\end{array}\quad$
$\begin{array}{ll}
e_2\bot e_2=e_1+e_2.
\end{array}$

${Trias}_{3}^6$ :	
$\begin{array}{ll}
e_1\dashv e_3=e_2,\\
e_3\dashv e_1=e_2,\\
e_3\dashv e_3=e_2,\\
\end{array}\,\,$
$\begin{array}{ll}
e_1\vdash e_1=e_2,\\
e_1\vdash e_3=e_2,\\
\end{array}\quad$
$\begin{array}{ll}
e_3\vdash e_1=e_2,\\
e_3\vdash e_3=e_2,\\
\end{array}$
$\begin{array}{ll}
e_3\bot e_1=e_2,\\
e_3\bot e_3=e_2.\\
\end{array}$\\

${Trias}_{3}^7$ :	
$\begin{array}{ll}
e_1\dashv e_1=e_2+e_3,\\
\end{array}\,\,$
$\begin{array}{ll}
e_1\vdash e_1=e_2+e_3,\\
\end{array}\quad$
$\begin{array}{ll}
e_1\bot e_1=e_2+e_3.
\end{array}$\\

${Trias}_{3}^8$ :	
$\begin{array}{ll}
e_2\dashv e_2=e_1,\\
e_2\dashv e_3=e_1,\\
e_3\dashv e_2=e_1,\\
\end{array}\,\,$
$\begin{array}{ll}
e_3\dashv e_3=e_1,\\
e_2\vdash e_2=e_1,\\
e_2\vdash e_3=e_1,\\
\end{array}\quad$
$\begin{array}{ll}
e_3\vdash e_2=e_1,\\
e_2\bot e_2=e_1,\\
\end{array}$
$\begin{array}{ll}
e_2\bot e_3=e_1,\\
e_3\bot e_2=e_1.\\
\end{array}$\\

${Trias}_{3}^9$ :	
$\begin{array}{ll}
e_2\dashv e_2=e_3,
\end{array}\,\,$
$\begin{array}{ll}
e_2\vdash e_2=e_3,\\
\end{array}\quad$
$\begin{array}{ll}
e_2\bot e_1=e_1+e_3,\\
e_2\bot e_2=e_1+e_3.\\
\end{array}$\\

${Trias}_{3}^{10}$ :	
$\begin{array}{ll}
e_2\dashv e_2=e_1+e_3,\\
\end{array}\,\,$
$\begin{array}{ll}
e_2\vdash e_2=e_1+e_3,\\
\end{array}\quad$
$\begin{array}{ll}
e_2\bot e_2=e_1+e_3.
\end{array}$\\

${Trias}_{3}^{11}$ :	
$\begin{array}{ll}
e_2\dashv e_1=e_3,\\
e_2\dashv e_2=e_3,\\
\end{array}\,\,$
$\begin{array}{ll}
e_2\vdash e_1=e_3,\\
e_2\vdash e_2=e_3,
\end{array}\quad$
$\begin{array}{ll}
e_2\bot e_1=e_3,\\
e_2\bot e_2=e_3.
\end{array}$

${Trias}_{3}^{12}$ :	
$\begin{array}{ll}
e_1\dashv e_1=e_3,\\
e_1\dashv e_2=e_3,\\
\end{array}\,\,$
$\begin{array}{ll}
e_2\dashv e_1=e_3,\\
e_1\vdash e_2=e_3,
\end{array}\quad$
$\begin{array}{ll}
e_2\vdash e_1=e_3,\\
e_2\vdash e_2=e_3.
\end{array}$
$\begin{array}{ll}
e_2\bot e_1=e_3,\\
e_2\bot e_2=e_3.
\end{array}$

\begin{proof}
 Let $\mathcal{T}$ be a three-dimensional vector space. To determine a triassociative algebra structure on $\mathcal{T}$ , we consider $\mathcal{T}$ with respect to one associative trialgebra operation.
Let $\mathcal{B}=(\mathcal{T}, \dashv)$ be the algebra\\
$$e_2 \dashv e_1 = e_3, \: e_2 \dashv e_2 = e_3$$
The multiplication operations $\vdash,\bot \:in\: \mathcal{T}$. We use the same method of the Proof of the Theorem \ref{the1}.

Then $\mathcal{B}=(\mathcal{T}, \dashv)$ it is isomorphic to ${Trias}_3^{11}$. The other associative trialgebras of the list of Theorem \ref{the2}
can be obtained by minor modification of the observation above.
\end{proof}

\begin{theorem}\label{the3}
 Any 4-dimensional real associative trialgebra either is associative or isomorphic to one of the following pairwise non-isomorphic associative tialgebras:
\end{theorem}

${Trias}_{4}^1$ :	
$\begin{array}{ll}
e_1\dashv e_1=e_2+e_4,\\
e_1\dashv e_3=e_2+e_4,\\
\end{array}\,\,$
$\begin{array}{ll}
e_3\dashv e_1=e_4,\\
e_1\vdash e_1=e_2+e_4,\\
\end{array}\quad$
$\begin{array}{ll}
e_1\vdash e_3=e_2+e_4,\\
e_3\vdash e_1=e_4,\\
\end{array}$
$\begin{array}{ll}
e_1\bot e_1=e_2+e_4,\\
e_1\bot e_3=e_4,\\
e_3\bot e_3=e_2,\\
\end{array}$\\

${Trias}_{4}^2$ :	
$\begin{array}{ll}
e_1\dashv e_1=e_2+e_4,\\
e_1\dashv e_3=e_2+e_4,\\
e_3\dashv e_1=e_2+e_4,\\
\end{array}\,\,$
$\begin{array}{ll}
e_1\vdash e_1=e_2+e_4,\\
e_1\vdash e_3=e_2+e_4,\\
e_3\vdash e_1=e_2+e_4,\
\end{array}\quad$
$\begin{array}{ll}
e_1\bot e_1=e_2+e_4,\\
e_1\bot e_3=e_2+e_4,\\
\end{array}$
$\begin{array}{ll}
e_3\bot e_1=e_2+e_4,\\
e_3\bot e_3=e_2.\\
\end{array}$\\

${Trias}_{4}^3$ :	
$\begin{array}{ll}
e_1\dashv e_1=e_2+e_4,\\
e_1\dashv e_3=e_2+e_4,\\
e_3\dashv e_1=e_2+e_4,\\
\end{array}\,\,$
$\begin{array}{ll}
e_1\vdash e_1=e_2+e_4,\\
e_1\vdash e_3=e_2+e_4,\\
\end{array}\quad$
$\begin{array}{ll}
e_3\vdash e_1=e_2+e_4,\\
e_1\bot e_1=e_2+e_4,\\
\end{array}$
$\begin{array}{ll}
e_1\bot e_3=e_2+e_4,\\
e_3\bot e_3=e_4.
\end{array}$\\

${Trias}_{4}^4$ :	
$\begin{array}{ll}
e_1\dashv e_2=e_4,\\
e_2\dashv e_1=e_4,\\
\end{array}\,\,$
$\begin{array}{ll}
e_2\dashv e_2=e_4,\\
e_2\vdash e_1=e_4,\\
\end{array}\quad$
$\begin{array}{ll}
e_2\vdash e_2=e_4,\\
e_3\vdash e_1=e_4,\\
\end{array}$
$\begin{array}{ll}
e_1\bot e_2=e_4,\\
e_2\bot e_1=e_4,\\
\end{array}$
$\begin{array}{ll}
e_1\bot e_2=e_4,\\
e_2\bot e_1=e_4,\\
\end{array}$\\

${Trias}_{4}^5$ :	
$\begin{array}{ll}
e_1\dashv e_2=e_4,\\
e_2\dashv e_1=e_4,\\
\end{array}\,\,$
$\begin{array}{ll}
e_2\dashv e_2=e_4,\\
e_2\vdash e_1=e_4,\\
\end{array}\quad$
$\begin{array}{ll}
e_2\vdash e_2=e_4,\\
e_3\vdash e_1=e_4,\\
\end{array}$
$\begin{array}{ll}
e_1\bot e_1=e_4,\\
e_2\bot e_1=e_4,\\
\end{array}$
$\begin{array}{ll}
e_3\bot e_3=e_4.
\end{array}$\\

${Trias}_{4}^6$ :	
$\begin{array}{ll}
e_3\dashv e_4=e_1+e_2,\\
e_4\dashv e_3=e_1+e_2,\\
\end{array}\,\,$
$\begin{array}{ll}
e_4\dashv e_4=e_1+e_2,\\
e_3\vdash e_4=e_1+e_2,\\
\end{array}\quad$
$\begin{array}{ll}
e_4\vdash e_3=e_1+e_2,\\
e_4\vdash e_4=e_1+e_2,\\
\end{array}$
$\begin{array}{ll}
e_3\bot e_4=e_1+e_2,\\
e_4\bot e_3=e_1+e_2,\\
e_4\bot e_4=e_1+e_2.
\end{array}$

${Trias}_{4}^7$ :	
$\begin{array}{ll}
e_2\dashv e_2=e_1+e_3,\\
e_2\dashv e_4=e_1+e_3,\\
\end{array}\,\,$
$\begin{array}{ll}
e_4\dashv e_2=e_1+e_3,\\
e_2\vdash e_2=e_1+e_3,\\
\end{array}\quad$
$\begin{array}{ll}
e_2\vdash e_4=e_1+e_3,\\
e_4\vdash e_2=e_1+e_3,\\
e_2\bot e_2=e_1+e_3,\\
\end{array}$
$\begin{array}{ll}
e_2\bot e_4=e_1+e_3,\\
e_4\bot e_2=e_1+e_3.\\
\end{array}$

${Trias}_{4}^8$ :	
$\begin{array}{ll}
e_2\dashv e_2=e_1+e_3,\\
e_2\dashv e_4=e_1+e_3,\\
\end{array}\,\,$
$\begin{array}{ll}
e_4\dashv e_2=e_1+e_3,\\
e_4\dashv e_4=e_1+e_3,\\
e_2\vdash e_2=e_1+e_3,\\
\end{array}\quad$
$\begin{array}{ll}
e_2\vdash e_4=e_1+e_3,\\
e_4\vdash e_2=e_1+e_3,\\
e_2\bot e_2=e_1+e_3,\\
\end{array}$
$\begin{array}{ll}
e_2\bot e_4=e_1+e_3,\\
e_4\bot e_4=e_1+e_3.\\
\end{array}$

${Trias}_{4}^9$ :	
$\begin{array}{ll}
e_2\dashv e_2=e_1+e_3,\\
e_2\dashv e_4=e_1+e_3,\\
e_4\dashv e_2=e_1+e_3,\\
\end{array}\,\,$
$\begin{array}{ll}
e_4\dashv e_4=e_1+e_3,\\
e_2\vdash e_2=e_1+e_3,\\
e_2\vdash e_4=e_1+e_3,\\
\end{array}\quad$
$\begin{array}{ll}
e_4\vdash e_2=e_1+e_3,\\
e_4\vdash e_4=e_3,\\
e_2\bot e_2=e_1+e_3,\\
\end{array}$
$\begin{array}{ll}
e_2\bot e_4=e_1+e_3,\\
e_4\bot e_2=e_1+e_3,\\
e_4\bot e_4=e_1
\end{array}$

${Trias}_{4}^{10}$ :	
$\begin{array}{ll}
e_2\dashv e_2=e_1+e_3,\\
e_2\dashv e_4=e_1+e_3,\\
\end{array}\quad$
$\begin{array}{ll}
e_4\dashv e_4=e_1+e_3,\\
e_2\vdash e_2=e_1+e_3,\\
\end{array}$
$\begin{array}{ll}
e_4\vdash e_2=e_1+e_3,\\
e_4\vdash e_4=e_3,\\
\end{array}$
$\begin{array}{ll}
e_2\bot e_4=e_1+e_3,\\
e_4\bot e_2=e_1+e_3,\\
e_4\bot e_4=e_1.
\end{array}$

${Trias}_{4}^{11}$ :	
$\begin{array}{ll}
e_2\dashv e_2=e_1+e_3,\\
e_2\dashv e_4=e_1+e_3,\\
\end{array}\quad$
$\begin{array}{ll}
e_4\dashv e_2=e_1+e_3,\\
e_2\vdash e_2=e_1+e_3,\\
\end{array}$
$\begin{array}{ll}
e_2\vdash e_4=e_1+e_3,\\
e_4\vdash e_2=e_1+e_3,\\
\end{array}$
$\begin{array}{ll}
e_2\bot e_2=e_1+e_3,\\
e_2\bot e_4=e_1+e_3,\\
e_4\bot e_2=e_1+e_3.\\
\end{array}$

${Trias}_{4}^{12}$ :	
$\begin{array}{ll}
e_2\dashv e_2=e_1+e_3,\\
e_2\dashv e_4=e_1+e_3,\\
e_4\dashv e_2=e_1+e_3,\\
\end{array}\quad$
$\begin{array}{ll}
e_4\dashv e_4=e_1+e_3,\\
e_2\vdash e_2=e_1+e_3,\\
\end{array}$
$\begin{array}{ll}
e_2\vdash e_4=e_1+e_3,\\
e_4\vdash e_2=e_1+e_3,\\
\end{array}$
$\begin{array}{ll}
e_2\bot e_2=e_1+e_3,\\
e_2\bot e_4=e_1+e_3,\\
e_4\bot e_2=e_1+e_3.\\
\end{array}$

${Trias}_{4}^{13}$ :	
$\begin{array}{ll}
e_2\dashv e_1=e_4,\\
e_2\dashv e_2=e_4,\\
\end{array}\quad$
$\begin{array}{ll}
e_3\dashv e_3=e_4,\\
e_1\vdash e_3=e_4,\\
\end{array}$
$\begin{array}{ll}
e_2\vdash e_2=e_4,\\
e_3\vdash e_1=e_4,\\
\end{array}$
$\begin{array}{ll}
e_1\bot e_1=e_4,\\
e_1\bot e_3=e_4,\\
\end{array}$
$\begin{array}{ll}
e_3\bot e_3=e_4.
\end{array}$\\

${Trias}_{4}^{14}$ :	
$\begin{array}{ll}
e_2\dashv e_2=e_1+e_3,\\
e_2\dashv e_4=e_1+e_3,\\
e_4\dashv e_2=e_1+e_3,\\
\end{array}\,\,$
$\begin{array}{ll}
e_4\dashv e_4=e_1+e_3,\\
e_2\vdash e_2=e_1+e_3,\\
e_2\vdash e_4=e_1+e_3,\\
\end{array}\quad$
$\begin{array}{ll}
  e_4\vdash e_2=e_1+e_3,\\
e_4\vdash e_4=e_3,\\
e_2\bot e_2=e_1+e_3,\\
\end{array}$
$\begin{array}{ll}
e_2\bot e_4=e_1+e_3,\\
e_4\bot e_2=e_1+e_3,\\
e_4\bot e_4=e_1.
\end{array}$

${Trias}_{4}^{15}$ :	
$\begin{array}{ll}
e_2\dashv e_1=e_3,\\
e_2\dashv e_2=e_3,\\
e_4\dashv e_1=e_3,\\
\end{array}\,\,$
$\begin{array}{ll}
e_4\dashv e_2=e_3,\\
e_1\vdash e_1=e_3,\\
e_1\vdash e_4=e_3,\\
\end{array}\quad$
$\begin{array}{ll}
  e_2\vdash e_1=e_3,\\
e_2\vdash e_4=e_3,\\
e_2\bot e_1=e_3,\\
\end{array}$
$\begin{array}{ll}
e_2\bot e_2=e_3,\\
e_4\bot e_1=e_3,\\
e_4\bot e_4=e_3.
\end{array}$

${Trias}_{4}^{16}$ :	
$\begin{array}{ll}
e_1\dashv e_1=e_2+e_4,\\
e_3\dashv e_1=e_2+e_4,\\
e_3\dashv e_3=e_2+e_4,\\
\end{array}\,\,$
$\begin{array}{ll}
  e_1\vdash e_1=e_2+e_4,\\
e_1\vdash e_3=e_2+e_4,\\
e_3\vdash e_1=e_2+e_4,\\
\end{array}\quad$
$\begin{array}{ll}
  e_3\vdash e_3=e_4,\\
e_1\bot e_1=e_2,\\
e_1\bot e_3=e_4,\\
\end{array}$
$\begin{array}{ll}
e_3\bot e_1=e_2,\\
e_3\bot e_3=e_2+e_4.
\end{array}$

\begin{proof}
 Let $\mathcal{T}$ be a three-dimensional vector space. To determine a triassociative algebra structure on $\mathcal{T}$ , we consider $\mathcal{T}$ with
respect to one associative trialgebra operation.
Let $\mathcal{C}=(\mathcal{T}, \dashv)$ be the algebra\\
$$e_1 \dashv e_2 = e_4, \: e_2 \dashv e_1 = e_4,e_2 \dashv e_2 = e_4.$$
The multiplication operations $\vdash,\bot \:in\: \mathcal{T}$.  We use the same method of the Proof of the Theorem \ref{the1}.

Then $\mathcal{C}=(\mathcal{T}, \dashv)$ it is isomorphic to ${Trias}_4^4$. The other associative trialgebras of the list of Theorem \ref{the3} can be obtained by
minor modification of the observation.
\end{proof}

\section{Derivations of low-dimensional associative trialgebras}
\begin{definition}\label{dia2}
A derivation of the associative trialgebras $\mathcal{T}$ is a linear transformation : $\mathcal{D} : \mathcal{T} \rightarrow \mathcal{T}$ satisfying
\begin{eqnarray}
d(x\dashv  y)=d(x)\dashv  y+x\dashv d(y)\\
d(x\vdash y)=d(x)\vdash y+x \vdash d(y)\\
d(x\bot y)=d(x)\bot y+x\bot d(y)
\end{eqnarray}
for all $x, y\in \mathcal{T}.$
 \end{definition}
The set of all derivations of $\mathcal{T}$ denoted by $Der\mathcal{D}$. It is clear that  $Der\mathcal{D}$ is a linear subspace of
$End\mathcal{D}.$

The coefficients of the above linear combinations $\left\{\mathcal{\gamma}_{ij}^k, \mathcal{\delta}_{rs}^t , \mathcal{\phi}_{pl}^m\right\}$
are called the structure constants of $\mathcal{T}$ on the basis $\left\{e_1, e_2,\ldots, e_n\right\}.$

A derivation being a linear transformation of the vector space $\mathcal{T}$ is represented in a matrix form  $\left[d_{ij}\right]_{ij=1,2\ldots,n}$ \,i.e\,
$\mathcal{D}(e_i)=\sum_{j=1}^nd_{ji}e_j\quad i=1,2,\ldots,n.$

According to the definition of the derivation, the entries $d_{ji}e_j\quad i=1,,2,\ldots,n$, of the matrix  $\left[d_{ij}\right]_{ij=1,2\ldots,n}$ must satisfy
the following systems of equations :

\begin{equation}\label{eq2}
\left\{\begin{array}{c}
\begin{array}{ll}
\sum_{q=1}^n(\gamma_{ij}^kd_{qk}- d_{ki}\gamma_{kj}^q- d_{kj}\gamma_{ik}^q)=0,\quad i,j,q\in \left\{1,n\right\}\\
\sum_{q=1}^n(\delta_{ij}^kd_{qk}- d_{ki}\delta_{kj}^q- d_{kj}\delta_{ik}^q)=0,\quad i,j,q\in \left\{1,n\right\}\\
\sum_{q=1}^n(\phi_{ij}^kd_{qk}-  d_{ki}\phi_{kj}^q- d_{kj}\phi_{ik}^q)=0,\quad i,j,q\in \left\{1,n\right\}.
\end{array}
\end{array}\right.
\end{equation}

\begin{theorem}
The derivations of two-dimensional associative trialgebras has the following form :
\end{theorem}

\begin{tabular}{||c||c||c||c||c||c||c||c||c||c||c||c||}
\hline
IC&Der$(\mathcal{D})$ &$Dim(\mathcal{D})$&IC&Der$(\mathcal{D})$&$Dim(\mathcal{D})$\\
			\hline
${Trias}_2^3$&
$\left(\begin{array}{cccc}
d_{11}&0\\
0&0
\end{array}
\right)$
&
1
&
${Trias}_2^7$&
$\left(\begin{array}{cccc}
d_{11}&0\\
0&d_{11}
\end{array}
\right)$
&
1
\\ \hline
${Trias}_2^4$&
$\left(\begin{array}{cccc}
d_{11}&0\\
-d_{11}&0
\end{array}
\right)$
&
1
&
${Trias}_2^8$&
$\left(\begin{array}{cccc}
0&d_{21}\\
0&-\frac{a-b}{b}d_{21}
\end{array}
\right)$
&
1
\\ \hline
${Trias}_2^6$&
$\left(\begin{array}{cccc}
0&0\\
0&d_{22}
\end{array}
\right)$
&
1
&
&
&
\\ \hline
\end{tabular}

\vspace{.2in}
\begin{theorem}\label{dthieo1}
The derivations of three-dimensional associative trialgebras has the following form :
\end{theorem}

\begin{tabular}{||c||c||c||c||c||c||c||c||c||c||c||c||}
\hline
IC&Der$(\mathcal{D})$ &$Dim(\mathcal{D})$&IC&Der$(\mathcal{D})$&$Dim(\mathcal{D})$\\
			\hline
${Trias}_3^3$&
$\left(\begin{array}{ccc}
0&0&0\\
d_{21}&0&d_{13}\\
0&0&d_{13}
\end{array}
\right)$
&
2
&
${Trias}_3^9$&
$\left(\begin{array}{ccc}
d_{11}&0&d_{11}\\
d_{21}&d_{23}&0\\
0&0&0
\end{array}
\right)$
&
3
\\ \hline
${Trias}_3^5$
&
$\left(\begin{array}{ccc}
d_{11}&0&d_{13}\\
-d_{11}&0&-d_{13}\\
-d_{11}&0&-d_{13}
\end{array}
\right)$
&
2
&
${Trias}_3^{10}$&
$\left(\begin{array}{ccc}
d_{11}&0&d_{13}\\
d_{21}&k_{10}&d_{23}\\
d_{31}&0&k_{11}
\end{array}
\right)$
&5
\\ \hline
${Trias}_3^6$&
$\left(\begin{array}{ccc}
0&0&0\\
0&d_{22}&0\\
0&0&0\\
\end{array}
\right)$
&
1
&
${Trias}_3^{11}$&
$\left(\begin{array}{ccc}
d_{11}&0&d_{13}\\
d_{21}&k_{13}\\
0&0&k_2
\end{array}
\right)$
&4
\\ \hline
${Trias}_3^{8}$&
$\left(\begin{array}{ccc}
d_{11}&0&0\\
d_{21}&\frac{1}{2}d_{11}&0\\
d_{31}&0&\frac{1}{2}d_{11}
\end{array}
\right)$
&
3
&
${Trias}_3^{12}$
&
$\left(\begin{array}{ccc}
d_{11}&0&d_{13}\\
0&d_{11}&d_{23}\\
0&0&2d_{11}
\end{array}
\right)$
&
3
\\ \hline
\end{tabular}

\begin{proof}
From Theorem \ref{dthieo1}, we provide the proof only for one case to illustrate the approach used, the other cases can be carried out similarly with or no
modification(s). Let consider ${Trias}_3^{3}$. Applying the systems of equations (\ref{eq2}). we get
$d_{11}=d_{12}=d_{13}=d_{31}=d_{22}=d_{32}=0,\quad d_{23}=d_{13},\quad d_{33}=d_{13}.$
Hence, the derivations of ${Trias}_3^{3}$ are given as follows\\
$d_1=\left(\begin{array}{ccc}
0&0&0\\
1&0&0\\
0&0&0
\end{array}
\right)$,\,$d_2=\left(\begin{array}{ccc}
0&0&1\\
0&0&1\\
0&0&0
\end{array}
\right)$ is basis of $Der(\mathcal{D})$ and Dim$Der(\mathcal{D})=2.$ The centroids of the remaining parts of three-dimension  associative trialgebras can be carried out in similar manner as shown above.
\end{proof}

\begin{theorem}\label{dthieo2}
The derivations of four-dimensional associative trialgebras has the following form :
\end{theorem}

\begin{tabular}{||c||c||c||c||c||c||c||c||c||c||c||c||}
\hline
IC&Der$(\mathcal{D})$ &$Dim(\mathcal{D})$&IC&Der$(\mathcal{D})$&$Dim(\mathcal{D})$\\
			\hline
${Trias}_4^1$&
$\left(\begin{array}{cccc}
d_{11}&d_{12}&0&d_{14}\\
0&d_{11}&0&0\\
0&d_{32}&d_{11}&d_{34}\\
0&0&0&2d_{11}\\
\end{array}
\right)$
&
5
&
${Trias}_4^9$
&
$\left(\begin{array}{cccc}
d_{11}&0&0&0\\
d_{21}&\frac{1}{2}d_{11}&d_{23}&0\\
0&0&d_{11}&0\\
d_{41}&0&d_{43}&\frac{1}{2}d_{11}
\end{array}
\right)$
&5
\\ \hline
${Trias}_4^2$
&
$\left(\begin{array}{cccc}
0&0&0&0\\
0&0&0&0\\
0&d_{32}&0&d_{34}\\
0&0&0&0\\
\end{array}
\right)$
&
2
&
${Trias}_4^{10}$&
$\left(\begin{array}{cccc}
d_{11}&0&0&0\\
d_{21}&\frac{1}{2}d_{11}&d_{23}&0\\
0&0&d_{11}&0\\
d_{41}&0&d_{43}&\frac{1}{2}d_{11}
\end{array}
\right)$
&5
\\ \hline
${Trias}_4^3$&
$\left(\begin{array}{cccc}
d_{11}&d_{12}&0&d_{14}\\
0&d_{11}&0&0\\
0&d_{32}&d_{11}&d_{34}\\
0&0&0&2d_{11}
\end{array}
\right)$
&
5
&
${Trias}_4^{11}$&
$\left(\begin{array}{cccc}
d_{11}&0&0&0\\
d_{21}&\frac{1}{2}d_{11}&d_{23}&0\\
0&0&d_{11}&0\\
d_{41}&0&d_{43}&\frac{1}{2}d_{11}
\end{array}
\right)$
&
5
\\ \hline
${Trias}_4^{4}$&
$\left(\begin{array}{cccc}
d_{11}&0&0&d_{14}\\
0&d_{11}&0&d_{24}\\
0&0&d_{11}&d_{34}\\
0&0&0&d_{44}
\end{array}
\right)$
&
5
&
${Trias}_4^{12}$
&
$\left(\begin{array}{cccc}
d_{11}&0&d_{8}&0\\
d_{21}&t_{7}&d_{23}&0\\
d_{31}&0&t_6&0\\
d_{41}&0&d_{43}&t_{7}
\end{array}
\right)$
&5
\\ \hline
${Trias}_4^5$&
$\left(\begin{array}{cccc}
d_{11}&0&0&d_{14}\\
0&d_{11}&0&d_{24}\\
0&0&d_{11}&d_{34}\\
0&0&0&2d_{11}
\end{array}
\right)$
&
5
&
${Trias}_4^{13}$&
$\left(\begin{array}{cccc}
d_{11}&0&0&d_{14}\\
0&d_{11}&0&d_{24}\\
0&0&d_{11}&d_{34}\\
0&0&0&2d_{11}
\end{array}
\right)$
&4
\\ \hline
${Trias}_4^6$&
$\left(\begin{array}{cccc}
d_{11}&d_{12}&0&0\\
d_{21}&t_1&0&0\\
d_{31}&d_{32}&t_2&0\\
d_{41}&d_{42}&0&t_3
\end{array}
\right)$
&
5
&
${Trias}_4^{14}$&
$\left(\begin{array}{cccc}
d_{11}&0&0&0\\
d_{21}&\frac{1}{2}d_{11}&d_{23}&0\\
0&0&d_{11}&0\\
d_{41}&0&d_{43}&\frac{1}{2}d_{11}
\end{array}
\right)$
&4
\\ \hline
${Trias}_4^{7}$&
$\left(\begin{array}{cccc}
0&0&0&0\\
d_{21}&0&d_{41}&0\\
0&0&0&0\\
d_{41}&0&0&0
\end{array}
\right)$
&
2
&
${Trias}_4^{15}$&
$\left(\begin{array}{cccc}
d_{11}&0&d_{13}&0\\
0&d_{11}&d_{23}&0\\
0&0&d_{11}&0\\
0&0&d_{43}&d_{11}
\end{array}
\right)$
&
4
\\ \hline
${Trias}_4^{8}$&
$\left(\begin{array}{cccc}
d_{11}&0&d_{13}&0\\
d_{21}&t_5&d_{23}&0\\
d_{31}&0&t_6&0\\
d_{41}&0&t_7&0
\end{array}
\right)$
&6
&
${Trias}_4^{16}$&
$\left(\begin{array}{cccc}
d_{11}&d_{12}&0&d_{14}\\
0&2d_{11}&0&0\\
0&d_{32}&d_{11}&d_{43}\\
0&0&0&2d_{11}
\end{array}
\right)$
&5
\\ \hline
\end{tabular}


\begin{proof}
From Theorem \ref{dthieo2}, we provide the proof only for one case to illustrate the approach used, the other cases can be carried out similarly with or no
modification(s). Let's consider ${Trias}_4^{1}$. Applying the systems of equations (\ref{eq3}). we get
$d_{21}=d_{31}=d_{41}=d_{42}=d_{21}=d_{31}=d_{41}=d_{42}=0,$
$ d_{22}=d_{11},\quad d_{33}=d_{11}\quad d_{44}=2d_{11}.$
Hence, the derivations of ${Trias}_3^{12}$ are given as follows\\
$d_1=\left(\begin{array}{cccc}
1&0&0&0\\
0&1&0&0\\
0&0&1&0\\
0&0&0&2
\end{array}
\right)$,\,$d_2=\left(\begin{array}{cccc}
0&1&0&0\\
0&0&0&0\\
0&0&0&0\\
0&0&0&0
\end{array}
\right)$ \,$d_3=\left(\begin{array}{cccc}
0&0&0&0\\
0&0&0&0\\
0&1&0&0\\
0&0&0&0
\end{array}
\right)$,\,$d_4=\left(\begin{array}{cccc}
0&0&0&1\\
0&0&0&0\\
0&0&0&0\\
0&0&0&0
\end{array}
\right)$,\,$d_5=\left(\begin{array}{cccc}
0&0&0&0\\
0&0&0&0\\
0&0&0&1\\
0&0&0&0
\end{array}
\right)$
is basis of $Der(\mathcal{T})$ and Dim$Der(\mathcal{T})=5.$ The centroids of the remaining parts of dimension three associative trialgebras can be carried out in a
similar manner as shown above.
\end{proof}

\vspace{.2in}
\begin{corollary}\,
\begin{itemize}
	\item The dimensions of the derivations of two-dimensional associative trialgebras vary between zero and one.
	\item The dimensions of the derivations of three-dimensional associative trialgebras vary between zero and five.
	\item The dimensions of the derivations of four-dimensional associative trialgebras vary between zero  and seven.
\end{itemize}
\end{corollary}
\begin{remark}
\noindent In the tables above the following notations are used :
\begin{itemize}
	\item $k_1=d_{11}+d_{21},\quad k_{2}=\frac{1}{2}d_{11}+\frac{1}{2}d_{31},\quad k_{3}=d_{11}+d_{31}-d_{13},\quad k_{4}=d_{11}+d_{13}-d_{22}.$
	\item $t_1=d_{11}+d_{21}-d_{12},\,\,t_2=\frac{1}{2}d_{11}+\frac{1}{2}d_{21},\,\,t_3=d_{11}-d_{21}-d_{12},\,\, t_4=d_{11}+d_{21},\,\,t_5=\frac{1}{2}d_{11}+\frac{1}{2}d_{31},$
	\item $t_6=d_{11}+d_{31}-d_{13},\,\,t_7=d_{11}+d_{31},\,\,t_8=d_{22}-d_{42},\,\, t_9=d_{11}-d_{31}-d_{13}.$
\end{itemize}
\end{remark}

\section{Centroids of low-dimensional associative trialgebras}
\subsection{Properties of centroids of associative trialgebras}
In this section, we declare the following results on properties of centroids of associative trialgebra $\mathcal{T}$.
\begin{definition}
Let $\mathcal{H}$ be a nonempty subset of $\mathcal{T}$. The subset
\begin{equation}
Z_{\mathcal{T}}(\mathcal{H})=\left\{x\in\mathcal{H} | x\bullet \mathcal{H} = \mathcal{H}\bullet x=0\right\},
\end{equation}
is said to be the centralized of $\mathcal{H}$ in $\mathcal{T}$ where the $\bullet$ is $\dashv, \vdash$ and $\bot,$ respectively.
\end{definition}

\begin{definition}\label{Cia}
Let $\mathcal{T}$ be an arbitrary associative trialgebra over a field $\mathbb{K}$.The left, right and middle centroids
$\Gamma_{\mathbb{K}}^\dashv(\mathcal{T}),\quad\Gamma_{\mathbb{K}}^\vdash(\mathcal{T})$ and $\Gamma_{\mathbb{K}}^\bot(\mathcal{T})$ of $\mathcal{T}$ are the spaces of $\mathbb{K}$-linear
transformations on  $\mathcal{T}$ given by
\begin{equation}
\Gamma_{\mathbb{K}}^\bullet(\mathcal{T})=\left\{\psi\in End_{\mathbb{K}}(\mathcal{T})|\psi(x\bullet y)=x\bullet\psi(y)=\psi(x)\bullet y\quad \text{for all}\quad x, y\in \Gamma\right\},
\end{equation}
where the $\bullet$ is $\dashv, \vdash$ and $\bot$ respectively.
\end{definition}

\begin{definition}
Let $\psi\in End(\mathcal{T})$. If $\psi(\mathcal{T})\subseteq Z(\mathcal{T})$ and $\psi(\mathcal{T}^2)=0$ then $\psi$ is called a central derivation.
The set of all central derivations of $\mathcal{T}$ is  denoted by $\mathbb{C}(\mathcal{T})$.
\end{definition}

\begin{proposition}
Consider $(\mathcal{T}, \vdash, \dashv, \bot)$ be an associative trialgebra. Then
\begin{enumerate}
	\item [i)]$\Gamma(\mathcal{T})Der(\mathcal{T})\subseteq Der(\mathcal{T})$.
		\item [ii)]$\left[\Gamma(\mathcal{T}), Dr(\mathcal{T})\right]\subseteq\Gamma(\mathcal{T}).$
	\item [iii)]$\left[\Gamma(\mathcal{T}), \Gamma(\mathcal{T})\right](\mathcal{T})\subseteq \Gamma(\mathcal{T})$ and $\left[\Gamma(\mathcal{T}), \Gamma(\mathcal{T})\right](\mathcal{T}^2)=0.$
\end{enumerate}
 \end{proposition}
\begin{proof}
The proof
 of parts $i)-iii)$ is strainghtforward by using definitions of derivation and centroid.
\end{proof}

\vspace{.2in}
\begin{proposition}
Let $\mathcal{T}$ be an associative trialgebra and $\varphi\in \Gamma_{\mathcal{T}},\, d\in Der(\mathcal{T}).$ Then $\varphi\circ d$ is a derivation of $\mathcal{T}.$
\end{proposition}
\begin{proof}
Indeed, if $x, y\in \mathcal{T}$, then
$$\begin{array}{ll}
(\varphi\circ d)(x\bullet y)
&= \varphi(d(x)\bullet y+x\bullet d(y))\\
&= \varphi(d(x)\bullet y)+\varphi(x\bullet d(y))=(\varphi\circ d)(x)\bullet y+x\bullet(\varphi\circ d)(y)
\end{array}$$
where the $\bullet$ is $\dashv, \vdash$ and $\bot$ respectively.
\end{proof}

\begin{proposition}
Let $\mathcal{T}$ be an associative trialgebra over a field $\mathbb{K}$. Then $\mathbb{C}(\mathcal{T}=\Gamma(\mathcal{T})\cap Der(\mathcal{T}).$
\end{proposition}

\begin{proof}
If $\psi\in \Gamma(\mathcal{T})\cap Der(\mathcal{T})$ the by definition of $\Gamma(\mathcal{T})$ and $Der(\mathcal{T})$ we have

$\psi(x\bullet y)=\psi(x)\bullet y+x\bullet\psi(y)$ and $\psi(x\bullet y)=\psi(x)\circ y=x\circ\psi(y)$ for $x,y\in \mathcal{T}.$ The yieds $\psi(\mathcal{T}\mathcal{T})=0$
and $\psi(\mathcal{T})\subseteq  Z(\mathcal{T})$ i.e $\Gamma(\mathcal{T})\cap Der(\mathcal{T})\subseteq \mathbb{C}(\mathcal{T}).$
The inverse is obvious since $\mathbb{C}(\mathcal{T})$ is in both $\Gamma(\mathcal{T})$ and $Der(\mathcal{T}),$ where the $\bullet$ is $\dashv, \vdash$ and $\bot$ respectively.
\end{proof}

\begin{proposition}
Let $(\mathcal{T}, \vdash, \dashv, \bot)$ be an associative trialgebra. Then  for any $d\in Der(\mathcal{T})$ and $\varphi\in \Gamma(\mathcal{T})$.
\begin{enumerate}
	\item [(i)] The composition $d\circ\varphi$ is in $\Gamma(\mathcal{T})$ if and only if $\varphi\circ d$ is a central derivation of $\mathcal{T}.$
		\item [(ii)] The  composition $d\circ\varphi$ is a derivation of $\mathcal{T}$ if and only if $\left[d,\varphi\right]$ is a central derivation of $\mathcal{T}.$
\end{enumerate}
\end{proposition}

\begin{proof}
\begin{enumerate}
	\item [i)]For any $\varphi\in \Gamma(\mathcal{T}),\, d\in Der(\mathcal{T}),\, \forall\,x,y\in \mathcal{T}.$ We have
	$$\begin{array}{ll}
d\circ\varphi(x\bullet y)=d\circ\varphi(x)\bullet y
&=d\circ\varphi(x)\bullet y+\varphi(x)\bullet d(y)\\
&=d\circ\varphi(x)\bullet y+\varphi\circ d(x\bullet y)-\varphi\circ d(x)\bullet y.
\end{array}$$
Therefore $(d\circ\varphi-\varphi\circ d)(x\bullet y)=(d\bullet\varphi-\varphi\circ d)(x)\bullet y.$
	\item [ii)] Let $d\circ\varphi\in Der(\mathcal{T})$, using $\left[d,\varphi\right]\in\Gamma(\mathcal{T}$, we get
	\begin{equation}\label{eq1}
	\left[d,\varphi\right](x\bullet y)=(\left[d, \varphi\right](x))\bullet y=x\bullet(\left[d,\varphi\right](y))
	\end{equation}
	On the other hand $\left[d, \varphi\right]d\circ\varphi-\varphi\circ d$ and $d\circ\varphi, \varphi\circ d\in Der(\mathcal{T}).$
	Therefore,	
	\begin{equation}\label{eq2}
\left[d, \varphi\right](x\bullet y)=(d(\varphi\circ(x))\bullet y+x\bullet(d\circ\varphi(y))-(\varphi\circ d(x))\bullet y-x\bullet(\varphi\circ d(y)).
\end{equation}
Due to (\ref{eq1}) and (\ref{eq2}) we get $x\bullet(\left[d, \varphi\right])(y)=(\left[d, \varphi\right])(x)\bullet y=0.$

\noindent Let's now $\left[d, \varphi\right]$ be a central derivation of $\mathcal{T}$. Then
$$\begin{array}{ll}
d\circ\varphi(x\bullet y)
&=\left[d\circ\varphi\right](x\bullet y)+(\varphi\circ d)(x\bullet y)\\
&=\varphi(\circ d(x)\bullet y)+\varphi(x\bullet d(y))\\
&=(\varphi\circ d)(x)\bullet y+x\bullet(\varphi\circ d)(y),
\end{array}$$
\end{enumerate}
where $\bullet$ represents the products $\dashv, \vdash$ and $\bot$ respectively.
\end{proof}

\subsection{Centroids of low-dimensional associative trialgebras}
We devoted to the description of centroids of two, three and four-dimensional associative trialgebra $\mathcal{T}.$
An element $\psi$ of the centroid $\Gamma(\mathcal{T})$ being a linear transformation of the vector space $\mathcal{T}$ is represented in matrix form $\left[\mathcal{T}_{ij}\right]_{i,j=1,\ldots,n}$ i.e
$\psi(e_i)=\sum_{j=1}^na_{ji}e_j\quad i=1,\ldots,n.$

According to the definition of the centroid the entries $\mathcal{T}_{ij}\quad i,j=1,2,\ldots, n$ of the matrix $\left[\mathcal{T}_{ij}\right]_{i,j=1,\ldots,n}$
must satisfy the following systems of equations :
\begin{equation}\label{eq3}
\left\{\begin{array}{clcl}
\begin{array}{ll}
\sum_{k=1}^n(\gamma_{ij}^kC_{pk}-C_{ki}\gamma_{kj}^p)=0,\quad \sum_{k=1}^n\gamma_{ij}^kC_{pk}-C_{kj}\gamma_{ik}^p)=0,\quad i,j,p\in \left\{1,n\right\}\\
\sum_{k=1}^n(\delta_{rs}^tC_{qt}-C_{tr}\delta_{ts}^q)=0,\quad \sum_{k=1}^n(\delta_{rs}^tC_{qt}-C_{ts}\delta_{rt}^q)=0,\quad r,s,q\in \left\{1,n\right\}\\
\sum_{k=1}^n(\psi_{pq}^rC_{sr}-C_{rp}\psi_{rq}^s)=0,\quad \sum_{k=1}^n(\psi_{pq}^rC_{sr}-C_{rq}\psi_{pr}^s)=0,\quad p,q,s\in \left\{1,n\right\}.
\end{array}
\end{array}\right.
\end{equation}

\vspace{.1in}
\begin{theorem}
The centroids of 2-dimensional complex associative trialgebra are given as follows :
\end{theorem}

\begin{tabular}{||c||c||c||c||c||c||c||c||c||c||c||c||}
\hline
IC&Centroid of Trias &$Dim(\mathcal{T})$&IC&Centroid of Trias&$Dim(\mathcal{T})$\\
			\hline
${Trias}_2^1$&
$\left(\begin{array}{cccc}
a_{11}&0\\
0&a_{11}
\end{array}
\right)$
&
1
&
${Trias}_2^5$&
$\left(\begin{array}{cccc}
a_{11}&0\\
0&a_{11}
\end{array}
\right)$
&
1
\\ \hline
${Trias}_2^2$&
$\left(\begin{array}{cccc}
a_{11}&0\\
0&a_{11}
\end{array}
\right)$
&
1
&
${Trias}_2^6$&
$\left(\begin{array}{cccc}
a_{11}&0\\
0&a_{11}
\end{array}
\right)$
&
1
\\ \hline
${Trias}_2^3$&
$\left(\begin{array}{cccc}
a_{11}&0\\
0&a_{11}
\end{array}
\right)$
&
1
&
${Trias}_2^7$&
$\left(\begin{array}{cccc}
a_{11}&0\\
0&a_{11}
\end{array}
\right)$
&
1
\\ \hline
${Trias}_2^4$&
$\left(\begin{array}{cccc}
a_{11}&0\\
0&a_{11}
\end{array}
\right)$
&
1
&
${Trias}_2^8$&
$\left(\begin{array}{cccc}
a_{11}&0\\
0&a_{11}
\end{array}
\right)$
&
1
\\ \hline
\end{tabular}

\vspace{.2in}
\begin{theorem}\label{cthieo1}
The centroids of 3-dimensional complex associative trialgebra are given as follows :
\end{theorem}

\begin{tabular}{||c||c||c||c||c||c||c||c||c||c||c||c||}
\hline
IC&Centroid of Trias &$Dim(\mathcal{T})$&IC&Centroid of Trias&$Dim(\mathcal{T})$\\
			\hline
${Trias}_3^1$&
$\left(\begin{array}{cccc}
a_{11}&0&0\\
0&a_{11}&0\\
0&0&a_{11}
\end{array}
\right)$
&
1
&
${Trias}_3^7$&
$\left(\begin{array}{cccc}
a_{11}&0&0\\
0&a_{11}&0\\
0&0&a_{11}
\end{array}
\right)$
&
1
\\ \hline
${Trias}_3^2$&
$\left(\begin{array}{cccc}
a_{11}&0&0\\
0&a_{11}&0\\
0&0&a_{11}
\end{array}
\right)$
&
1
&
${Trias}_3^8$&
$\left(\begin{array}{cccc}
a_{11}&a_{12}&a_{13}\\
0&a_{22}&a_{13}\\
0&k_3&k_4
\end{array}
\right)$
&
4
\\ \hline
${Trias}_3^3$&
$\left(\begin{array}{cccc}
a_{11}&0&0\\
a_{21}&a_{6}&0\\
0&0&a_{11}
\end{array}
\right)$
&
2
&
${Trias}_3^9$&
$\left(\begin{array}{cccc}
a_{11}&0&0\\
0&a_{11}&0\\
0&0&a_{11}
\end{array}
\right)$
&
1
\\ \hline
${Trias}_3^4$&
$\left(\begin{array}{cccc}
a_{11}&0&0\\
0&a_{11}&0\\
a_{31}&0&a_{11}
\end{array}
\right)$
&
2
&
${Trias}_3^{10}$&
$\left(\begin{array}{cccc}
a_{11}&0&0\\
0&a_{11}&a_{23}\\
0&0&a_{11}
\end{array}
\right)$
&
2
\\ \hline
${Trias}_3^5$&
$\left(\begin{array}{cccc}
a_{11}&0&a_{13}\\
a_{21}&k_1&-a_{13}\\
a_{21}&0&k_2
\end{array}
\right)$
&
3
&
${Trias}_3^{11}$&
$\left(\begin{array}{cccc}
a_{11}&0&a_{13}\\
a_{21}&k_3&-a_{23}\\
a_{31}&0&k_4
\end{array}
\right)$
&
5
\\ \hline
${Trias}_3^{6}$&
$\left(\begin{array}{cccc}
a_{11}&0&0\\
a_{21}&a_{11}&0\\
0&0&a_{11}
\end{array}
\right)$
&
2.
&
${Trias}_{12}^{3}$
&
$\left(\begin{array}{ccc}
a_{11}&0&a_{13}\\
0&a_{11}&a_{23}\\
0&0&2a_{11}
\end{array}
\right)$
&
3
\\ \hline
\end{tabular}

\vspace{.2in}
\begin{proof}
From Theorem \ref{cthieo1}, we provide the proof only for one case to illustrate the approach used, the other cases can be carried out similarly with or no
modification(s). Let's consider ${Trias}_3^{12}$. Applying the systems of equations (\ref{eq3}). we get
$a_{12}=a_{21}=a_{31}=a_{32}=0,\quad a_{22}=a_{11},\quad a_{33}=2a_{11}.$
Hence, the derivations of ${Trias}_3^{12}$ are given as follows\\
$a_1=\left(\begin{array}{ccc}
1&0&0\\
0&0&0\\
0&1&2
\end{array}
\right)$,\,$a_2=\left(\begin{array}{ccc}
0&0&1\\
0&0&0\\
0&0&0
\end{array}
\right)$ \,$a_3=\left(\begin{array}{ccc}
0&0&0\\
0&0&1\\
0&0&0
\end{array}
\right)$ is basis of $Der(\Gamma)$ and Dim$Der(\Gamma)=3.$ The centroids of the remaining parts of dimension three associative trialgebras can be carried out in a
similar manner as above.
\end{proof}

\begin{theorem}
The centroids of 4-dimensional associative trialgebra are given as follows :
\end{theorem}

\begin{tabular}{||c||c||c||c||c||c||c||c||c||c||c||c||}
\hline
IC&Centroid of Trias &$Dim(\mathcal{T})$&IC&Centroid of Trias&$Dim(\mathcal{T})$\\
			\hline
${Trias}_4^1$&
$\left(\begin{array}{cccc}
a_{11}&0&0&a_{14}\\
0&a_{11}&0&0\\
0&0&a_{11}&a_{34}\\
0&0&0&a_{11}
\end{array}
\right)$
&
3
&
${Trias}_4^8$&
$\left(\begin{array}{cccc}
a_{11}&0&a_{13}&0\\
a_{21}&k_{3}&0&0\\
a_{31}&0&k_{4}&0\\
a_{41}&0&a_{43}&k_{3}
\end{array}
\right)$
&
7
\\ \hline
${Trias}_4^2$&
$\left(\begin{array}{cccc}
a_{11}&0&0&a_{14}\\
0&a_{22}&0&0\\
0&a_{32}&0&a_{34}\\
0&0&a_{11}&a_{11}
\end{array}
\right)$
&
5
&
${Trias}_4^9$&
$\left(\begin{array}{cccc}
a_{11}&0&0&0\\
a_{21}&a_{22}&a_{23}&0\\
0&0&a_{11}&0\\
a_{41}&0&a_{43}&a_{11}
\end{array}
\right)$
&
6
\\ \hline
${Trias}_4^3$&
$\left(\begin{array}{cccc}
a_{11}&a_{12}&0&a_{14}\\
0&a_{11}&0&0\\
0&a_{32}&a_{11}&a_{34}\\
0&0&0&a_{11}
\end{array}
\right)$
&
5
&
${Trias}_4^{10}$&
$\left(\begin{array}{cccc}
a_{11}&0&0&0\\
a_{21}&a_{11}&a_{23}&0\\
0&0&a_{11}&0\\
a_{41}&0&a_{43}&a_{11}
\end{array}
\right)$
&
5
\\ \hline
${Trias}_4^4$&
$\left(\begin{array}{cccc}
a_{11}&0&0&a_{14}\\
0&a_{11}&0&a_{24}\\
0&0&a_{11}&a_{34}\\
0&0&0&a_{11}
\end{array}
\right)$
&
4
&
${Trias}_4^{11}$&
$\left(\begin{array}{cccc}
a_{11}&0&0&0\\
a_{21}&a_{11}&a_{23}&0\\
0&0&a_{11}&0\\
a_{41}&0&a_{43}&a_{11}
\end{array}
\right)$
&
5
\\ \hline
${Trias}_4^5$&
$\left(\begin{array}{cccc}
a_{11}&0&0&0\\
0&a_{11}&0&0\\
0&0&a_{11}&0\\
0&0&0&a_{11}
\end{array}
\right)$
&
1
&
${Trias}_{4}^{12}$&
$\left(\begin{array}{cccc}
a_{11}&0&a_{13}&0\\
a_{21}&k_{2}&a_{23}&0\\
a_{31}&0&k_{4}&0\\
a_{41}&0&a_{43}&k_3
\end{array}
\right)$
&
5
\\ \hline
${Trias}_{4}^6$&
$\left(\begin{array}{cccc}
a_{11}&a_{12}&0&0\\
a_{21}&k_{1}&0&0\\
a_{31}&a_{32}&k_{2}&0\\
a_{41}&a_{42}&0&k_2
\end{array}
\right)$
&
7
&
${Trias}_{4}^{13}$
&
$\left(\begin{array}{cccc}
a_{11}&0&0&a_{14}\\
0&a_{11}&0&a_{24}\\
0&0&a_{11}&a_{34}\\
0&0&0&a_{11}
\end{array}
\right)$
&
4
\\ \hline
${Trias}_{4}^7$&
$\left(\begin{array}{cccc}
a_{11}&0&a_{13}&0\\
a_{21}&k_{2}&a_{23}&0\\
a_{31}&0&k_{4}&0\\
a_{41}&0&a_{43}&k_3
\end{array}
\right)$
&
6
&
${Trias}_{4}^{14}$
&
$\left(\begin{array}{cccc}
a_{11}&0&0&0\\
a_{21}&a_{11}&a_{23}&0\\
0&0&a_{11}&0\\
a_{41}&0&a_{43}&a_{11}
\end{array}
\right)$
&
5
\\ \hline
${Trias}_{4}^{15}$&
$\left(\begin{array}{cccc}
a_{11}&0&0&0\\
0&a_{11}&0&0\\
0&0&a_{11}&0\\
0&0&0&a_{11}
\end{array}
\right)$
&
1
&
${Trias}_{4}^{16}$
&
$\left(\begin{array}{cccc}
a_{11}&a_{12}&0&a_{14}\\
0&a_{11}&0&0\\
0&a_{32}&a_{11}&a_{34}\\
0&0&0&a_{11}
\end{array}
\right)$
&
5
\\ \hline
\end{tabular}

\begin{proof}
From Theorem \ref{cthieo1}, we provide the proof only for one case to illustrate the approach used, the other cases can be carried out similarly with or no modification(s). Let's consider ${Trias}_4^{16}$. Applying the systems of equations (\ref{eq3}). we get
$a_{21}=a_{31}=a_{41}=a_{42}=a_{21}=a_{31}=a_{41}=a_{42}=0,$
$ a_{22}=a_{11},\quad a_{33}=a_{11}\quad a_{44}=a_{11}.$
Hence, the derivations of ${Trias}_3^{12}$ are given as follows\\
$a_1=\left(\begin{array}{cccc}
1&0&0&0\\
0&1&0&0\\
0&0&1&0\\
0&0&0&1
\end{array}
\right)$,\,$a_2=\left(\begin{array}{cccc}
0&1&0&0\\
0&0&0&0\\
0&0&0&0\\
0&0&0&0
\end{array}
\right)$ \,$a_3=\left(\begin{array}{cccc}
0&0&0&0\\
0&0&0&0\\
0&1&0&0\\
0&0&0&0
\end{array}
\right)$,\,$a_4=\left(\begin{array}{cccc}
0&0&0&1\\
0&0&0&0\\
0&0&0&0\\
0&0&0&0
\end{array}
\right)$,\,$a_5=\left(\begin{array}{cccc}
0&0&0&0\\
0&0&0&0\\
0&0&0&1\\
0&0&0&0
\end{array}
\right)$
is basis of $Der(\Gamma)$ and Dim$Der(\Gamma)=5.$ The centroid of the remaining parts of dimension three associative trialgebras can be carried out in a
similar manner as shown above.
\end{proof}

\begin{corollary}\,
\begin{itemize}
	\item The dimensions of the centroids of two-dimensional associative trialgebras are one.
	\item The dimensions of the centroids of two-dimensional associative trialgebras vary between one and five.
	\item The dimensions of the centroids of two-dimensional associative trialgebras vary between one and seven.
\end{itemize}
\end{corollary}
\begin{remark}
\noindent In the tables above the following notations are used :
\begin{itemize}
  \item \textbf{IC} : Isomorphism classes of associative trialgebras.
	\item \textbf{Dim} : Dimensions of the associative trialgebras of derivations and centroids.
	\item $k_1=a_{11}+a_{21}-a_{12},\quad a_2=a_{11}+a_{21},\quad k_3=a_{11}+a_{31},\quad k_4=a_{11}+a_{31}-a_{13}.$
\end{itemize}
\end{remark}

\end{document}